\documentclass[12pt]{amsart}

\usepackage{pdfsync}

\usepackage{amssymb}
\usepackage[colorlinks=true, urlcolor=rltblue, citecolor=drkgreen, linkcolor=drkred] {hyperref}
\usepackage {hyperref}
\usepackage{color}
\definecolor{rltblue}{rgb}{0,0,0.2}
\definecolor{drkgreen}{rgb}{0,0.2,0}
\definecolor{drkred}{rgb}{0.2,0,0}
\newtheorem{thm}{Theorem}
\newtheorem{lemma}[thm]{Lemma}
\newtheorem{proposition}[thm]{Proposition}

\newtheorem{theorem}[thm]{Theorem}

\newtheorem{fact}[thm]{Fact}

\theoremstyle{definition}
\newtheorem{definition}[thm]{Definition}

\theoremstyle{remark}

\newtheorem{obs}[thm]{Observation}
\newtheorem{remark}[thm]{Remark}
\newtheorem{example}[thm]{Example}

\newtheorem{historic}[thm]{Historic Remark}

\theoremstyle{plain}

\newcounter{contenumi}

\def\leqt{\leq_T}
\def\geqt{\geq_T}

\def\teq{\equiv_T}
\def\D{{\mathcal D}}

\def\O{{\mathcal O}}
\def\Q{\mathcal{Q}}

\def\upto{\mathop{\upharpoonright}}
\def\la{\langle}
\def\ra{\rangle}
\def\and{\mathrel{\&}}

\def\Si{\Sigma}

\def\A{\mathcal{A}}
\def\B{\mathcal{B}}

\def\F{\mathcal{F}}

\def\om{\omega}

\def\si{\sigma}

\def\Gm{G_\mathbf{m}}
\def\Glip{G_\mathbf{lip}}
\def\Gtm{\tilde{G}_\mathbf{m}}
\def\Yp{Y^{\sf p}}

\def\A{\mathcal A}

\def\Si{\Sigma}

\def\om{\omega}

\def\implies{\Rightarrow}

\def\Afr{\mathfrak{A}}
\def\Bfr{\mathfrak{B}}

\def\Afr{{\mathfrak A}}
\def\Bfr{{\mathfrak B}}

\newcommand{\fr}{{}^{\smallfrown}}
\newcommand{\prim}{{\rm PRec}}

\newcommand{\mcone}{^\triangledown_\mathbf{m}}

\title{The uniform Martin's conjecture for many-one degrees}

\author{Takayuki Kihara}
\thanks{The first-named author was partially supported by a Grant-in-Aid for JSPS fellows.}

\email{kihara@math.berkeley.edu}
\urladdr{\href{http://www.math.berkeley.edu/~kihara/index.html}{www.math.berkeley.edu/$\sim$kihara}}

\author{Antonio Montalb\'an}
\thanks{The second-named author was partially supported by NSF grant DMS-0901169 and the Packard Fellowship.
}

\email{antonio@math.berkeley.edu}
\urladdr{\href{http://www.math.berkeley.edu/~antonio/index.html}{www.math.berkeley.edu/$\sim$antonio}}

\begin{document}

\maketitle


%

\begin{abstract}
We study functions from reals to reals which are uniformly degree-invariant from Turing-equivalence to many-one equivalence, and compare them ``on a cone.'' 
We prove that they are in one-to-one correspondence with the Wadge degrees, which can be viewed as a refinement of the uniform Martin's conjecture for uniformly invariant functions from Turing- to Turing-equivalence.

Our proof works in the general case of many-one degrees on $\Q^\om$ and Wadge degrees of functions $\om^\om\to\Q$ for any better quasi ordering $\Q$.
\end{abstract}

\section{Introduction}\label{sec:introduction}

The uniform version of Martin's conjecture for functions from Turing- to Turing-equivalence was proved by Slaman and Steel in \cite{Steel82, SlaSte88}.
We prove it for functions that are uniformly degree-invariant from Turing- to many-one-equivalence, getting a finer and richer structure.

Often in mathematics, and particularly in computability theory, we consider a large, complicated class of objects, among which very  few of those objects are {\em natural}, and where the class of natural objects behaves in a much better way than the whole class.
This can be disconcerting at times. 
But in some cases, the contrast between the general behavior and the behavior of natural objects can be quite interesting and intriguing.
In this paper, we consider the class of many-one degrees.

\begin{definition}
For sets $A,B\subseteq \om$, we say that $A$ is {\em many-one reducible} to $B$ (sometimes referred to as $m$-reducible and written $A\leq_m B$) if there is a computable function $\psi\colon\om\to\om$ such that $n\in A\iff \psi(n)\in B$ for all $n\in \om$.
As usual, from this pre-ordering we define an equivalence $\equiv_m$ and a degree structure referred to as the {\em m-degrees}.
\end{definition}

The many-one degrees have been widely studied in computability theory since its beginnings (see \cite[Chapters III and VI]{Odi89}).
There are quite a few natural $m$-degrees all computability theorists know: $\emptyset$, $\om$, the complete c.e.\ set (i.e., the $m$-degree of $0'$), the complete d.c.e. degree (i.e., the $m$-degree of $0'\times \bar{0'}$), the complete $\Sigma^0_2$ set, the $m$-degree of Kleene's $\O$, etc.
These are still very few compared to the whole set of $m$-degrees.
For instance, we know of no natural $m$-degree of a c.e.\ set that is neither complete nor computable, despite there being infinitely many such degrees.
We know of no natural $m$-degree of a $\Sigma^1_1$ set that is neither $\Si^1_1$-complete nor hyperarithmetic, again despite there being lots of them.
The general structure of the $m$-degrees is quite complex: there are continuum-size anti-chains; every countable poset embeds in it, even below $0'$ \cite{KP54}; its first-order theory is computably isomorphic to true second-order arithmetic \cite{NS80}; etc. (see \cite[Chapter VI]{Odi89}.)

We give a complete characterization of the {\em natural} many-one degrees.
In the same sense, a characterization of the natural Turing degrees is already well known and follows from the uniform Martin's conjecture,  which was proved by Slaman and Steel \cite{Steel82, SlaSte88}: The natural Turing degrees are, essentially, the iterates of the Turing jump through the transfinite.
Indeed, Becker \cite{Becker88} proved that the natural nonzero Turing degrees are exactly the ones obtained as the universal set of a reasonable pointclass up to relativization.
It turns out that the answer for the many-one degrees is richer: The natural many-one degrees are in one-to-one correspondence with the Wadge degrees.
Except for a few ideas we borrowed from the proof of the uniform Martin's conjecture, most of our argument is completely different.
Our results can be viewed as a refinement of the the uniform Martin's conjecture, as the jump of a natural Turing degree is a natural $m$-degree.
However, there are many natural $m$-degrees that are not distinguished by Turing equivalence.
Indeed, every natural Turing degree contains a lot of $m$-inequivalent natural $m$-degrees; for instance, the complete c.e.\ set is Turing equivalent to the complete d.c.e.\ set, though they are not $m$-equivalent.

We do not have a formal mathematical definition of what it means to be a {\em natural} $m$-degree.
Thus, there will have to be an empirical, non-mathematical claim in our argument:
\begin{quote}
Natural $m$-degrees induce Turing-to-many-one, uniformly degree-invariant functions, as in Definition \ref{def: UI}.
\end{quote}
This claim comes from the observation that, in computability,  all proofs {\em relativize}, which is also empirically observed.
That is, for any given theorem, if we change the notion of computability by that of computability relative to an oracle $X$,  the resulting theorem can then still be proved using the same proof.
Furthermore, the notions we deal with in computability theory also relativize, and so do their properties. 
Thus, if we have a natural m-degree $\bf s$, we can associate to it a function that, given an oracle $X$, returns the relativization of $\bf s$ to $X$, denoted ${\bf s}^X$.
Furthermore, if we relativize to an oracle $Y\teq X$, the classes of partial $X$-computable functions and of partial $Y$-computable functions are the same, so we should obtain the same $m$-degrees.
We let the interested reader contemplate this fact further, and we will now move on to the purely mathematical results.

Here is the definition of the uniformly degree-invariant functions we mentioned above.

\begin{definition}\label{def: UI}
We say that a function $f\colon\om^\om\to 2^\om$ is {\em uniformly $(\leq_T,\leq_m)$-order preserving} (abbreviated  $(\leq_T,\leq_m)$-UOP) if, for every $X,Y\in \om^\om$,
\[
X\leq_T Y \implies f(X)\leq_m f(Y),
\]
and furthermore, there is a computable function $u\colon\om\to\om$ such that, for all $X,Y\in\om^\om$,
\[
X\leq_TY\mbox{ via }e\;\Longrightarrow\;f(X)\leq_m f(Y)\mbox{ via }u(e).
\]
(By $X\leq_T Y$ via $e$, we mean that it is the $e$-th Turing functional $\Phi_e$ that Turing reduces $X$ to $Y$, and  analogously with $m$-reducibility.)

We say that $f$ is {\em uniformly $(\equiv_T,\equiv_m)$-invariant} (abbreviated $(\equiv_T,\equiv_m)$-UI) if there is a computable function $u\colon\om^2\to\om^2$ such that, for all $X,Y\in\om^\om$,
\[
X\equiv_TY\mbox{ via }(i,j)\;\Longrightarrow\;f(X)\equiv_m f(Y)\mbox{ via }u(i,j).
\]
\end{definition}

There is a natural notion of {\em largeness} for sets of Turing degrees given by Martin's measures:
A Turing-degree-invariant set $\A\subseteq \om^\om$ has Martin measure 1 if it contains a {\em Turing cone}, i.e., a set of the form $\{X\in \om^\om: Y\geqt X\}$ for some $X\in \om^\om$, and has Martin measure 0 otherwise.
Martin proved that if determinacy holds for all sets in a class $\Gamma$, then this is a $\si$-additive measure on the degree-invariant sets in a pointclass $\Gamma$ \cite{Martin68}.
We use this notion of largeness to extend the {\em many-one ordering} to $(\equiv_T,\equiv_m)$-UI functions.

\begin{definition}\label{def:many-one-cone for 2}
For $A,B\subseteq\omega$ and an oracle $C\in \omega^\omega$, we say that {\em $A$ is many-one reducible to $B$ relative to $C$} (and write $A\leq_m^C B$) if there is a $C$-computable function $\Phi^C_e$ such that 
\[
(\forall n\in\om)\; n\in A\iff \Phi^C_e(n)\in B.
\]
Given $f,g\colon \om^\om\to 2^\om$, we say that {\em $f$ is many-one reducible to $g$ on a cone} (and write $f \leq\mcone g$) if 
\[(\exists C\in \om^\om)(\forall X\geqt C)\;f(X)\leq_m^C g(X).\]
\end{definition}

It is clear that $\leq\mcone$ is a pre-ordering and hence induces an equivalence on functions we denote by $\equiv\mcone$.
Our objective is to compare $\equiv\mcone$-degrees of $(\equiv_T,\equiv_m)$-UI functions with the Wadge degrees.

\begin{definition}[Wadge \cite{Wadge83}]
Given $\A,\B\subseteq \om^\om$, we say that $\A$ is {\em Wadge reducible} to $\B$ (and write $\A\leq_w\B$) if there is a continuous function $f\colon\om^\om\to\om^\om$ such that $X\in \A\iff f(X)\in \B$ for all $X\in \om^\om$.
\end{definition} 

Again, $\leq_w$ is a pre-ordering which induces an equivalence $\equiv_w$ and a degree structure.
The Wadge degrees are rather well-behaved, at least under enough determinacy.
If we assume $\Gamma$-determinacy, then the Wadge degrees of sets in $\Gamma$ are semi-well-ordered in the sense that they are well-founded and all anti-chains have size at most 2 (as proved by Wadge \cite{Wadge83}, and Martin and Monk).
Furthermore, they are all {\em natural}, and we can assign {\em names} to each of them using an ordinal and a symbol from $\{\Si,\Pi\}$ (see \cite{Wesep78}), a name from which we can understand the nature of that Wadge degree.

Here is our main theorem for the case of sets:

\begin{theorem}(AD$+$DC)   \label{thm: main for 2}
There is an isomorphism between the partial ordering of $\equiv\mcone$-degrees of $(\equiv_T,\equiv_m)$-UI functions  ordered by $\leq\mcone$ and the partial ordering of Wadge degrees of subsets of $\om^\om$ ordered by Wadge reducibility.
\end{theorem}

The definition of the isomorphism is not complicated (see Section \ref{sec: the plan}).
It is the proof that it is a correspondence that requires work.
We get the following simple corollaries. 
The clopen Wadge degrees correspond to the constant functions.
Then the open non-clopen Wadge degree corresponds to the $(\equiv_T,\equiv_m)$-UI function that gives the complete c.e.\ set.
Thus, there are no $(\equiv_T,\equiv_m)$-UI functions strictly in between the constant functions and the complete ones.
The Hausdorff-Kuratowski difference hierarchy of ${\bf \Delta}^0_2$ sets of reals corresponds to the Ershov hierarchy of $\Delta^0_2$ sets of natural numbers.
Thus, up to $\equiv\mcone$-equivalence, the only ${\bf \Delta}^0_2$ $(\equiv_T,\equiv_m)$-UI functions are the ones corresponding to the Ershov hierarchy.
The complete Wadge degree of the $\Si^1_1$ set of reals corresponds to the $(\equiv_T,\equiv_m)$-UI function given by the complement of the hyperjump.
Since every Wadge degree of a ${\bf \Si}^1_1$ set must be either ${\bf \Si}^1_1$-complete or Borel, we get that, up to $\equiv\mcone$-equivalence, a ${\bf \Si}^1_1$ $(\equiv_T,\equiv_m)$-UI function must be either complete or hyperarithmetic.

In our construction of Section \ref{sec: surjectivity}, we actually assign a $(\leq_T,\leq_m)$-UOP function to each Wadge degree.
Thus, our proof also gives the following theorem:

\begin{theorem}(AD$+$DC)  \label{thm: ui to upo for 2}
Every $(\equiv_T,\equiv_m)$-UI function $\om^\om\to 2^\om$ is $\equiv\mcone$-equivalent to a $(\leq_T,\leq_m)$-UOP one.
\end{theorem}

\subsection{The extension to better-quasi-orderings}

Our main theorem will actually be more general than Theorem \ref{thm: main for 2}.
A subset of $\om$ can be viewed as a function $\om\to 2$, and a subset of $\om^\om$ as a function $\om^\om\to 2$.
Instead, we will consider functions $\om\to \Q$ and $\om^\om\to \Q$, where $\Q$ is a better-quasi-ordering (bqo). 
The definition of better-quasi-ordering is complicated (Definition \ref{def:bqo}), so for now, let us just say that better-quasi-orderings are well-founded, have no infinite antichains, and have nice closure properties.

The generalizations of all the notions defined above are straightforward.
We include them for completeness.

\begin{definition}\label{def:many-one-cone}
Let $(\mathcal{Q};\leq_\mathcal{Q})$ be a quasi-ordered set.
For $A,B\in\mathcal{Q}^\omega$ and an oracle $C\in \om^\omega$, we say that {\em $A$ is $\Q$-many-one reducible to $B$ relative to $C$} (written $A\leq_m^C B$) if there is a $C$-computable function $\Phi_e^C\colon\om\to\om$ such that 
\[
(\forall n\in\om)\;A(n)\leq_\mathcal{Q}B(\Phi_e^C(n)).
\]
For functions $\om^\om\to \mathcal{Q}^\om$, the definitions of $(\leq_T,\leq_m)$-UOP, $(\leq_T,\leq_m)$-UI, and $\leq\mcone$ are then exactly as before, using the new notion of $\Q$-many-one reducibility.


For $\mathcal{Q}$-valued functions $\A,\B\colon\om^\om\to\mathcal{Q}$, we say that {\em $\A$ is $\Q$-Wadge reducible to $\B$} (written $\A\leq_w \B$) if there is a continuous function $\theta\colon\om^\om\to\om^\om$ such that 
\[
(\forall X\in\om^\om)\;\A(X)\leq_\mathcal{Q}\B(\theta(X)).
\]
\end{definition}

On the one hand, considering the general case does not add to the complexity of the proof --- the proofs for 2 and for general $\Q$ are essentially the same.
There are bqos  $\Q$ other than 2 for which the $\Q$-many-one degrees are interesting too. 
For $\Q=3$, the poset with three incomparable elements, Marks \cite{Marks} proved that many-one equivalence on $3^\om$ is a uniformly-universal countable Borel equivalence relation, while this is not the case for $2^\omega$.
Since $(\equiv_T,\equiv_m)$-UI functions are nothing more than uniform reductions from Turing- to many-one-equivalence, understanding such functions can shed light on the structure of countable, degree-invariant Borel equivalence relations.
For $\Q=(\om;\leq)$, we have that, for $f,g\colon\om\to\om$, $f\leq_m g$ if and only if there is a computable speed up of $g$ that grows faster than $f$, that is, if there is a computable $h\colon\om\to\om$ such that $g\circ h(n)\geq f(n)$ for all $n\in\om$.
On the side of the Wadge degrees, Steel showed that when $\Q$ is the class of ordinals, the Wadge degrees are well-founded.
In \cite{KMwadge}, the authors provide a full description of the Wadge degrees of $\Q$-valued Borel functions for each bqo $\Q$, extending work of Duparc \cite{Dup01,Dup03}, Selivanov \cite{Sel07}, and others.

Here is our main theorem:

\begin{theorem}(AD$^+$)   \label{thm: main for Q}
There is an isomorphism between the partial ordering of $\equiv\mcone$-degrees of $(\equiv_T,\equiv_m)$-UI functions $\om^\om \to\Q^\om$  ordered by $\leq\mcone$ and the partial ordering of $\Q$-Wadge degrees of functions $\om^\om\to \Q$ ordered by $\Q$-Wadge reducibility.
\end{theorem}

\begin{theorem}(AD$^+$)  \label{thm: ui to upo for Q}
Every $(\equiv_T,\equiv_m)$-UI function $\om^\om\to Q^\om$ is $\equiv\mcone$-equivalent to a $(\leq_T,\leq_m)$-UOP one.
\end{theorem}

\subsection{Background facts on $\Q$-Wadge degrees}

In 1970s, Martin and Monk showed that the Wadge degrees of subsets of $\om^\om$ are well-founded, and hence semi-well-ordered by Wadge's Lemma \cite{Wadge83}.
Steel then showed that the Wadge degrees of ordinal-valued functions with domain $\om^\om$ are well-ordered (see \cite[Theorem 1]{Dup03}).
Later, van Engelen--Miller--Steel \cite{EMS} employed bqo theory to unify these results, and they showed that if $\mathcal{Q}$ is bqo, then so are the Wadge degrees of $\mathcal{Q}$-valued Borel functions.
More recently, Block \cite{Blo14} introduced the notion of a {\em very strong better-quasi-order} to remove the Borel-ness assumption from van Engelen--Miller--Steel's theorem.
(We show in Section \ref{sec:set-theory} below, that under AD$^+$, bqos and very strong bqos are the same thing.)

To define bqos, we need to introduce some notation.
Let $[\om]^\om$ be the set of all strictly increasing sequences on $\om$, whose topology is inherited from $\om^\om$.
We also assume that a quasi-order $\mathcal{Q}$ is equipped with the discrete topology.
Given $X\in[\om]^\om$, by $X^-$ we denote the result of dropping the first entry from $X$ (or equivalently, $X^-=X\setminus\{\min X\}$, if we  think of $X\in[\om]^\om$ as an infinite subset of $\om$).

\begin{definition}[Nash-Williams \cite{NW68}]\label{def:bqo}
A quasi-order $\mathcal{Q}$ is called a {\em better-quasi-order} (abbreviated as bqo) if, for any continuous function $f\colon[\om]^\om\to\mathcal{Q}$,  there is $X\in[\om]^\om$ such that $f(X)\leq_\mathcal{Q}f(X^-)$.
\end{definition}

The formulation of the definition above is due to Simpson \cite{Sim85}.
It is not hard to prove that every bqo is also a {\em well-quasi-order} (often abbreviated as wqo), that is, that it is well-founded and has no infinite antichain.

\begin{example}
For a natural number $k$, the discrete order $\mathcal{Q}=(k;=)$, which we will denote by $k$, is a bqo.
More generally, every finite partial ordering is a bqo.
For $\Q=k$, the $\mathcal{Q}$-valued functions are called {\em $k$-partitions}.
\end{example}

Let us now state the key facts that we will be using about the $\Q$-Wadge degrees.
Special cases of the following facts were proved by van Engelen--Miller--Steel \cite[Theorem 3.2]{EMS} for Borel functions, and by Block \cite[Theorem 3.3.10]{Blo14} for very strong bqos $\mathcal{Q}$ under AD.
AD$^+$ proves the general result.

\begin{fact}\label{fact:Wadge-bqo}
(AD$^+$)
If $\mathcal{Q}$ is a bqo, then the Wadge degrees of $\mathcal{Q}$-valued functions on $\om^\om$ form a bqo too.
\end{fact}

There are two more facts about $\Q$-Wadge degrees that we will use throughout the paper.

\begin{definition}
We say that a $\mathcal{Q}$-Wadge degree $\mathbf{a}$ is {\em $\sigma$-join-reducible} if $\mathbf{a}$ is the least upper bound of a countable collection $(\mathbf{b}_i)_{i\in\om}$ of $\mathcal{Q}$-Wadge degrees such that $\mathbf{b}_i<_w\mathbf{a}$.
Otherwise, we say that $\mathbf{a}$ is {\em $\sigma$-join-irreducible}.
\end{definition}

The following fact gives a better way to characterize $\si$-join-reducibility.
Its proof uses the well-foundedness of the $\mathcal{Q}$-Wadge degrees, which is an immediate consequence of Fact \ref{fact:Wadge-bqo}.

\begin{fact}\label{fact:sji-deg-pres}
(AD$^+$)
Let $\mathcal{Q}$ be a bqo.
A function $\A\colon\om^\om\to\mathcal{Q}$ is $\sigma$-join-irreducible if and only if there is an $X\in\om^\om$ such that $\A\leq_w\A\upto[X\upto n]$ for every $n\in\om$.

A function $\A\colon\om^\om\to\mathcal{Q}$ is $\sigma$-join-reducible if and only if it is Wadge equivalent to a function of the form $\bigoplus_{n\in\om}\A_n$, where each $\A_n$ is $\si$-join-irreducible and $\A_n<_w\A$, and where $\bigoplus_{n\in\om}\A_n$ is defined by $(\bigoplus_{n\in\om}\A_n)(n\fr X)=\A_n(X)$.
\end{fact}

The third fact that we need is a generalization of Steel--van Wesep's theorem \cite{Wesep78} from $\Q=2$ to general $\Q$, proved by Block \cite{Blo14}.
The following generalization of self-duality is due to Louveau and Saint-Raymond \cite{LS90}.

\begin{definition}
We say that a function $\mathcal{A}\colon\om^\om\to\mathcal{Q}$ is {\em self-dual} if there is a continuous function $\theta\colon\om^\om\to\om^\om$ such that $\mathcal{A}(\theta(X))\not\leq_\mathcal{Q}\mathcal{A}(X)$ for all $X\in\om^\om$.
\end{definition}

Assuming AD, Block \cite[Proposition 3.5.4]{Blo14} showed the following fact for very strong bqos.
We get it for all bqos under AD$^+$:

\begin{fact}\label{fact:sjr-equal-sd}
(AD$^+$)
Let $\mathcal{Q}$ be a bqo.
Then a $\mathcal{Q}$-valued function on $\om^\om$ is self-dual if and only if it is $\sigma$-join-reducible.
\end{fact}

\subsection{The set-theoretic assumptions}
\label{sec:set-theory}

Our main theorems are stated under the assumption of AD$^+$, which is an extension of the axiom of determinacy  introduced by Woodin \cite{Woo99}.
If we want to assume less than AD$^+$, our results are still true for restricted classes of functions.
For instance, they are true for Borel functions just in ZFC, and true for projective functions if we assume DC+PD.

Let $\Gamma$ be a pointclass of sets of reals containing all Borel sets closed under countable unions, finite intersections, and continuous substitutions.
We concentrate on $\Gamma$-functions $f:\om^\om\to\mathcal{Q}$ whose range is countable, where a function $g:\om^\om\to\mathcal{Q}^\om$ can also be thought of as a function from $\om^{\om}\times\om$ ($\simeq\om^\om$) to $\mathcal{Q}$ in an obvious way (see also Definition \ref{def:uncurrying}).

For our results to hold for functions in $\Gamma$, we need to assume, first, that all Wadge-like games (introduced in Section \ref{sec: games}) for $\Gamma$-functions are determined, and second, that Facts \ref {fact:Wadge-bqo}, \ref{fact:sji-deg-pres}, and \ref{fact:sjr-equal-sd} hold for functions in $\Gamma$.
The first assertion is ensured by assuming that {\em all sets in $\Gamma$ are determined} whenever the ranges of functions are countable.
Our assumption of countability of the range is only used to ensure this part (and thus this restriction can be removed under AD).

We will now argue that assuming that {\em all sets in $\Gamma$ are Ramsey} gives us these three facts for any bqo $\mathcal{Q}$.
Note that this $\Gamma$-Ramsey hypothesis actually implies that all sets in $\Gamma$ are completely Ramsey (that is, all sets in $\Gamma$ have the Baire property with respect to the Ellentuck topology) under our assumption on $\Gamma$ (see Brendle-L\"owe \cite[Lemma 2.1]{BrLo99}).
Fact \ref{fact:sji-deg-pres} only uses well-foundedness of $\mathcal{Q}$-Wadge degrees of $\Gamma$-functions, which clearly follows from Fact \ref{fact:Wadge-bqo}, on top of ZFC.
For Facts \ref{fact:Wadge-bqo} and \ref{fact:sjr-equal-sd} we need the following observation.

\begin{obs}
Suppose that all sets in $\Gamma$ are determined and Ramsey, and let $\Q$ be a bqo.
We say that $\Q$ is a {\em $\Gamma$-bqo} if, for every $\Gamma$-function $f\colon[\om]^\om\to\mathcal{Q}$, there is $X\in[\om]^\om$ such that $f(X)\leq_\mathcal{Q}f(X^-)$.

Our assumption on $\Gamma$ implies that if $\Q$ is a bqo, it is also a $\Gamma$-bqo:
This is because every such $f$ in $\Gamma$ has the Baire property with respect to the Ellentuck topology by our assumption that all sets in $\Gamma$ are completely Ramsey.
Louveau-Simpson \cite{LoSi82} showed that, for every Ellentuck-Baire function $f:[\om]^\om\to\mathcal{Y}$ where $\mathcal{Y}$ is a metric space,  not necessarily separable, there exists an infinite set $X\subseteq \om$ such that $f$ is continuous when restricted to $[X]^\om$.
By applying this to the discrete metric space $\mathcal{Y}:=\mathcal{Q}$, the above argument verifies our claim.

One can then carry out the van Engelen--Miller--Steel proof \cite[Theorem 3.2]{EMS} for $\Gamma$ functions exactly as Block did in \cite[Theorem 3.3.10]{Blo14} to get that the $\Q$-Wadge degrees of functions in $\Gamma$ are bqo.
%
%
%
The argument only requires that $\Gamma$ is closed under countable (separated) union and continuous substitution.
Similarly, we can use Block's argument \cite[Theorem 3.4.4]{Blo14} to show the Steel-van Wesep Theorem \cite{Wesep78} for $\mathcal{Q}$-valued $\Gamma$-functions, that is, that Wadge self-duality and Lipschitz self-duality are equivalent for $\Gamma$-functions.
Fact \ref{fact:sjr-equal-sd} then follows from the standard argument (see \cite[Section 3]{Wesep78} or \cite[Proposition 3.5.4]{Blo14}) and Fact \ref{fact:sji-deg-pres}.
\end{obs}

Therefore, what we actually prove in this paper is the following:

\begin{quote}
If we assume that all sets in $\Gamma$ are determined and Ramsey, then our main Theorems \ref{thm: main for 2} and \ref{thm: main for Q} hold when restricted to $\Gamma$-functions whose range is countable.
\end{quote}

In particular, Theorems \ref{thm: main for 2}, \ref{thm: ui to upo for 2}, \ref{thm: main for Q}, and \ref{thm: ui to upo for Q}  for Borel functions can be proved in ZFC (since all Borel sets are determinied and Ramsey under ZFC \cite{Mar75,GaPr73}), and for projective functions can be proved under PD (since all projective sets are Ramsey under PD \cite{HaKe81}; indeed, $\mathbf{\Delta}^1_n$-determinacy implies that all $\mathbf{\Pi}^1_n$ sets are Ramsey for any positive even number $n$).
Our assumption AD$^+$ implies that all sets of reals are determined and Ramsey.

We also notice that our hypothesis that all $\Gamma$-sets are Ramsey is only used to ensure that every bqo is $\Gamma$-bqo.
For $\mathcal{Q}=2$, we can prove our main theorem without assuming the $\Gamma$-Ramsey hypothesis.
This is because the discrete ordered set $2=\{0,1\}$ is a very strong bqo (i.e., $\Gamma$-bqo for any $\Gamma$) within AD$+$DC (see Block \cite[Corollary 3.3.9]{Blo14}).
Indeed, Wadge's Lemma, Martin--Monk's Lemma, and Steel--van Wesep's Theorem are all provable in AD$+$DC, and these are all that we need to prove our main theorem.
This is the reason why we can state Theorems \ref{thm: main for 2} and \ref{thm: ui to upo for 2} only assuming AD$+$DC.

We will not mention these assumptions anymore through the rest of the paper.
The reader may either assume AD$^+$, or assume that we are only working with functions in a pointclass $\Gamma$ all of whose sets are determined and Ramsey.

\section{The Plan}\label{sec: the plan}

The mapping $\Afr$ that we will use to embed the $(\equiv_T,\equiv_m)$-UI functions into the $\Q$-Wadge degrees is quite simple.
The difficult part will be to prove that it actually gives a one-to-one correspondence.

\begin{definition}\label{def:uncurrying}
Given $f\colon\om^\om\to\mathcal{Q}^\om$, we define a function $\Afr(f) \colon \om^\om\to\mathcal{Q}$ as follows:
\[
\Afr(f)(n\fr X) = f(X)(n)
\]
for $n\in\om$ and $X\in\om^\om$. Here, $n\fr X$ is the concatenation of $n$ and $X$.
\end{definition}

This function will only work well on a subset of the $(\equiv_T,\equiv_m)$-UI functions, the $\Afr$-minimal functions, which we define below.
Before, we need to introduce the following notion:

\begin{definition}
Abusing notation, by {\em perfect tree} we mean a map $T[\cdot]\colon \om^{<\omega}\to \om^{<\omega}$ together with its image, satisfying $\si\subseteq \tau\iff T(\si)\subseteq T(\tau)$ for all $\si,\tau\in \om^{<\om}$.
For each $X\in \om^\om$, we can define $T[X]\in\om^\om$ in a obvious way; we often think of $T$ directly as a continuous map $T[\cdot]\colon \om^\omega\to \om^\omega$.
We use $[T]$ to denote the set $\{T[X]: X\in \om^\om\}$ of {\em paths} through $T$.

By a {\em uniformly pointed perfect tree} (abbreviated as {\em u.p.p.~tree}), we mean a perfect tree which is computable from each of its paths in a uniform way.
In other words, it is a perfect tree $T[\cdot]\colon \om^{<\omega}\to \om^{<\omega}$ such that there is an index $e$ such that $\Phi_e(Y)=T$ for any $Y\in[T]$.
\end{definition}

The main property of u.p.p.\ trees is that, for every $X\geqt T$, we have that $X \teq T[X]$, and we can compute the indices for this Turing equivalence given the index for $X\geqt T$.
Here is how u.p.p.\ trees interact with $(\equiv_T,\equiv_m)$-UI function.
In the statement of the lemma, we view the trees as maps $\om^\om\to\om^\om$.

\begin{lemma}\label{lem:UI-uppt-res}
Let $f \colon \om^\om\to\mathcal{Q}^\om$ be a $(\equiv_T,\equiv_m)$-UI function and let $S$ and $T$ be u.p.p.~trees.
\begin{enumerate}
\item If $S\leq_TT$, then $\Afr(f\circ T)\leq_w\Afr(f\circ S)$.    \label{UI-uppt-res, part 1}
\item If $f$ is $(\leq_T,\leq_m)$-UOP, then $\Afr(f\circ T)\equiv_w\Afr(f)$.                     \label{UI-uppt-res, part 3}
\item $f\circ T\equiv\mcone f$.                                                \label{UI-uppt-res, part 2}
\end{enumerate}
\end{lemma}

\begin{proof}
For (\ref{UI-uppt-res, part 1}), it is not hard to see that, since $S$ and $T$ are uniformly pointed and $S\leq_TT$, one can computably extract the triple $(S,T,X)$ from $T[X]$ and the pair $(T,X)$ from $S[T\oplus X]$ in a uniform manner.
Therefore, there is a pair of Turing reductions witnessing $T[X]\equiv_TS[T\oplus X]$ which does not depend on $X$.
Thus, since $f$ is $(\equiv_T,\equiv_m)$-UI, there is a computable function $\Psi$ such that
\[f(T[X])(n)\leq_\mathcal{Q}f(S[T\oplus X])(\Psi(n))\]
for any $n\in\om$.
Consequently, we have that
\[\Afr(f\circ T)(n\fr X)=\Afr(f)(n\fr T[X])\leq_\mathcal{Q}\Afr(f)(\Psi(n)\fr S[T\oplus X])=\Afr(f\circ S)(\Psi(n)\fr T\oplus X).\]

For (\ref{UI-uppt-res, part 3}), we only need to show that $\Afr(f\circ T)\geq_w\Afr(f)$, as the other reduction follows from (\ref{UI-uppt-res, part 1}).
There is an index that we can use to compute $X$ from $T[X]$ for all $X$, and hence there is a computable function $\psi$ witnessing $f(X)\leq_m f(T[X])$ for all $X$.
We then have 
\[
\Afr(f)(n\fr X) = f(X)(n)=f(T[X])(\psi(n))=\Afr(f\circ T)(\psi(n)\fr X).
\]

For (\ref{UI-uppt-res, part 2}), assume that $X\geq_TT$.
Then $X\equiv_TT[X]$, so let $(i,j)$ be a pair of indices witnessing this.
Let $u$ witness that $f$ is $(\equiv_T,\equiv_m)$-UI.
Then we have $f(X)(n)=f(T[X])(\Phi_{u(i,j)}(n))$ and $f(T[X])(n)=f(X)(\Phi_{u(j,i)}(n))$ for any $n\in\om$.
This clearly implies that $f\circ T\equiv\mcone f$.
\end{proof}

Since the $\Q$-Wadge degrees are well-founded (actually better-quasi-ordered by Fact \ref{fact:Wadge-bqo}), by Lemma \ref{lem:UI-uppt-res}, (\ref{UI-uppt-res, part 1}) we get that there is a $C$ such that the $\Q$-Wadge degree of $\Afr(f\circ T)$ is the same for all u.p.p.\ trees $T\geqt C$.

\begin{definition}
We say that $f\colon\om^\om\to\mathcal{Q}^\om$ is {\em $\Afr$-minimal} if for all u.p.p.\ trees $T$, $\Afr(f\circ T)\equiv_w\Afr(f)$.
\end{definition}

It follows from the lemma above that every $(\equiv_T,\equiv_m)$-UI function is $\equiv\mcone$-equivalent to an $\Afr$-minimal one, and that if $f$ is $(\leq_T,\leq_m)$-UOP, it is $\Afr$-minimal already.
We can thus concentrate only on the $\Afr$-minimal $(\equiv_T,\equiv_m)$-UI functions.

\begin{lemma}  \label{lem: Afr well defined}
Let $f,g\colon\om^\om\to\mathcal{Q}^\om$ be $(\equiv_T,\equiv_m)$-UI, {\em $\Afr$-minimal} functions.
Then  $f\leq\mcone g$ implies $\Afr(f)\leq_w \Afr(g)$.
\end{lemma}
\begin{proof}
There is a $X\in\om^\om$ such that, for each $X\geqt C$, there is some $e$ such that $\Phi^C_e$ is a many-one reduction $f(X)\leq^C_m g(X)$.
We then use Martin's Lemma (see \cite[Lemma 3.5]{MarSlaSte}), saying that if $\om^\om$ is partitioned into countably many subsets, then one of them contains all infinite paths through a u.p.p.\ tree, to obtain an index $e$ and a u.p.p.\ tree $T$ such that, for all $Y\in [T]$, $f(Y)\leq^C_m g(Y)$ via $\Phi_e$.
We thus get that, for all $X\in \om^\om$ and $n\in\om$, $f(T[X])(n)= g(T[X])(\Phi^C_e(n))$, and hence that $\Afr(f\circ T)\leq_w \Afr(g\circ T)$.
Since both $f$ and $g$ are $\Afr$-minimal, this implies $\Afr(f)\leq_w \Afr(g)$.
\end{proof}

We now have a well-defined map from the $\leq\mcone$-degrees of $(\equiv_T,\equiv_m)$-UI functions to the $\Q$-Wadge degrees: given a $(\equiv_T,\equiv_m)$-UI function $f$, let $g$ be a $(\equiv_T,\equiv_m)$-UI function that is $\Afr$-minimal and $\equiv\mcone$-equivalent to $f$, and let the image of the $\equiv\mcone$-degree of $f$ be the $\Q$-Wadge degree of $\Afr(g)$.
To show that this map is an isomorphism, i.e., Theorem \ref{thm: main for Q}, and to also get Theorem \ref{thm: ui to upo for Q}, we will show the following two propositions:

\begin{proposition}\label{prop: onto}
For every $\Q$-Wadge degree $\A$, there is a $(\leq_T,\leq_m)$-UOP function $g$ such that $\Afr(g)\equiv_w \A$.
\end{proposition}

\begin{remark}\label{rmk: standard form}
Let us say that $g$ {\em is in standard form} if either $\Afr(g)$ is non-self-dual, or it is of the form $\bigoplus g_n$, where $\Afr(g_n)$ is non-self-dual  for each $n$, where we define $\bigoplus_ng_n\colon\om^\om\to\Q^\om$ by $(\bigoplus_ng_n)(X)(\langle m,k\rangle)=g_m(X)(k)$.
It will follow from the proof of Proposition \ref{prop: onto} in the next section that we can assume $g$  is of the form $\bigoplus g_n$, and hence is in standard form.
We can then use Lemma \ref{lem:UI-uppt-res} to find an oracle $C$ such that, for all u.p.p.\ trees $S$, $\Afr(g_n\circ S)$ has minimal Wadge degree, and hence each of the $g_n$'s is $\Afr$-minimal.
\end{remark}

\begin{proposition}\label{prop: embeddability}
Let $f,g\colon\om^\om\to\mathcal{Q}^\om$ be $(\equiv_T,\equiv_m)$-UI, {\em $\Afr$-minimal} functions.
Then $f\leq\mcone g$ if and only if $\Afr(f)\leq_w \Afr(g)$.
\end{proposition}

We will prove Proposition \ref{prop: onto} and Remark \ref{rmk: standard form} in Section  \ref{sec: surjectivity}.
We will prove Proposition \ref{prop: embeddability} in Sections \ref{sec: games} and \ref{sec: wadge to m games}.

\section{Surjectivity} \label{sec: surjectivity}

The next step is to show that $\Afr$ is onto.
We devote this subsection to proving Proposition \ref{prop: onto}.

Given an oracle $C\in\om^\om$, a function $p\colon\om\to\om$ is said to be {\em $C$-primitive recursive} if it can be obtained by using the usual axioms of primitive recursive functions, including the function $n\mapsto C(n)$ in the list of initial functions.
A {\em primitive recursive functional} is a function $P \colon \om^\om\to\om^\om$ such that $P(C)$ is $C$-primitive recursive uniformly in $C$.
Let $(\prim_e)_{e\in\omega}$ be an effective list of all primitive recursive functionals from $\om^\om$ into $\om^\om$, so that $\prim\colon (e,X)\mapsto\prim_e(X)$ is computable.
We now introduce the following operation $\Bfr$ that will almost work as an inverse of $\Afr$.

\begin{definition}
Given $\A \colon \om^\om\to\mathcal{Q}$ and $C\in \om^\om$, let $\Bfr^C(\A)\colon \om^\om\to\mathcal{Q}^\om$ be defined by
\[
\Bfr^C(\A)(X)(e)= \A( \prim_e(C\oplus X)).
\]
\end{definition}

We will show that, for some large enough $C$, $\Bfr^C(\A)$ is $(\leq_T,\leq_m)$-UOP and that the $\equiv\mcone$-degree of $\Bfr^C(\A)$ is independent of $C$.
We start by showing that $\Bfr^C(\A)$ is always an inverse of $\A$, even if $\Bfr^C(\A)$ is not $(\equiv_T,\equiv_m)$-UI.

\begin{lemma}\label{lem:onto1}
For any $\A \colon \om^\om\to\mathcal{Q}$ and $C\in \om^\om$, we have $\Afr(\Bfr^C(\mathcal{A}))\equiv_w\mathcal{A}$.
\end{lemma}

\begin{proof}
Note that $\Afr(\Bfr^C(\mathcal{A}))(e\fr X)=\mathcal{A}(\prim_e(C\oplus X))$.
Let $i$ be an index of the function $C\oplus X\mapsto  X$, that is, $X=\prim_i(C\oplus X)$.
Then, given $X$, one can easily see that $\mathcal{A}(X)=\mathcal{A}(\prim_i(C\oplus X))=\Afr(\Bfr^C(\mathcal{A}))(i\fr X)$.
Thus, $\A\leq_w\Afr(\Bfr^C(\mathcal{A}))$.
For the other reduction, notice that the map $(e,X)\mapsto\prim_e(C\oplus X)$ is continuous, which witnesses that $\Afr(\Bfr^C(\mathcal{A}))\leq_w\mathcal{A}$.
\end{proof}

The following lemma shows that, when $\Bfr^C(\A)$ is $(\leq_T,\leq_m)$-UOP, $\Bfr^C(\A)$ always gives us the same function up to $\equiv\mcone$, independently of the oracle $C$.

\begin{lemma}\label{lem:onto2}
Let $\A \colon \om^\om\to\mathcal{Q}$, and $C,D\in\om^\om$.
If $\Bfr^C(\A)$ and $\Bfr^D(\A)$ are $(\leq_T,\leq_m)$-UOP, then $\Bfr^C(\A) \equiv\mcone \Bfr^D(\A)$.
\end{lemma}
\begin{proof}
It suffices to show that, for any $X\geq_TC\oplus D$, $\Bfr^C(\mathcal{A})(X)\leq_m\Bfr^D(\mathcal{A})(X)$ holds.
Let $v$ be such that $\prim_e(C\oplus X) = \prim_{v(e)}(D\oplus C\oplus X)$.
Note that $v$ gives us a many-one reduction
\[
\Bfr^C(\A)(X) \leq_m \Bfr^D(\A)(C\oplus X).
\]
For $X\geqt C$, since $C \oplus X \leqt X$ and $\Bfr^D(\A)$ is $(\leq_T,\leq_m)$-UOP, we get  that 
\[
\Bfr^D(\A)(C\oplus X)\leq_m \Bfr^D(\A)(X).
\]
We thus get $\Bfr^C(\A)(X)\leq_m  \Bfr^D(\A)(X)$, as needed.
The other inequality is analogous.
\end{proof}

What is left to show that is that $\Bfr^C(\A)$ is $(\leq_T,\leq_m)$-UOP for some $C$.
We will not get exactly this --- but close enough.
We start with the case when $\A$ is not self-dual, for which we first need to prove a quick lemma.
We say that a function $\theta\colon\om^\om\to\om^\om$ is {\em Lipschitz} if $\theta(X)\upto n$ depends only on $X\upto n$ for every $X\in \om^\om$, $n\in \om$; or in other words, if $X\upto n= Y\upto n\implies \theta(X)\upto n=\theta(Y)\upto n$.

\begin{lemma}
Let $\A \colon \om^\om\to\mathcal{Q}$ be not self-dual, $\B\colon \om^\om\to\mathcal{Q}$, and $\D\subseteq\om^\om$.
If there is a continuous function $\theta\colon \D\to\om^\om$ such that $\mathcal{B}(X)\leq_\mathcal{Q}\mathcal{A}(\theta(X))$ for all $X\in\mathcal{D}$, then there is a Lipchitz $\hat{\theta}\colon \om^\om\to\om^\om$ such that $\mathcal{B}(X)\leq_\mathcal{Q}\mathcal{A}(\hat{\theta}(X))$ for all $X\in\mathcal{D}$.
\end{lemma}
\begin{proof}
Consider the following variation of the Wadge game, which we denote by $G_{diag}(\A, \B\upto \D)$:
Players I and II choose $x_n,y_n\in\omega$ alternately, and produce $X=(x_n)_{n\in\om}$ and $Y=(y_n)_{n\in\om}$, respectively.
Player II wins if $Y\in\mathcal{D}$ and $\A(X)\not\geq_\Q\B(Y)$.
A winning strategy for II would give us a Lipchitz function $\Psi$ such that $\A(X)\not\geq_\Q\B(\Psi(X))$ for all $X\in\om^\om$.
Composing with $\theta$, we would then have that $\A(X)\not\geq_\Q\A(\theta\circ\Psi(X))$, contradicting that $\A$ is not self-dual.
Thus, Player I must have a winning strategy, which gives us a Lipchitz function $\hat{\theta}\colon \om^\om\to \om^\om$.
$\hat{\theta}$ must satisfy that, for all $X\in\D$, $\A(\hat{\theta}(X))\geq_\Q\B(X)$ as wanted. 
\end{proof}

As in the previous proof, we can always identify a winning strategy $\tau$ with a Lipchitz function $\theta_\tau$.
Moreover, $n\mapsto \tau(X\upto n)$ is $(\tau\oplus X)$-primitive recursive uniformly in $\tau\oplus X$.
In other words, there is a primitive recursive code $e$ such that, if $\tau$ defines a Lipschitz function $\theta_\tau$, then we have $\theta_\tau(X)=\prim_e(\tau\oplus X)$.

\begin{lemma}\label{lem:UOP-non-self-dual}
If $\A \colon \om^\om\to\mathcal{Q}$ is not self-dual, there exists $C$ such that $\Bfr^C(\A)$ is  $(\leq_T,\leq_m)$-UOP.
\end{lemma}
\begin{proof}

We will construct an oracle $C\in\om^\om$ and a computable function $q\colon\om\to\om$ such that, if $X\leq_TY$ via $\Phi_d$, then $\Bfr^C(\mathcal{A})(X)\leq_m\Bfr^C(\mathcal{A})(Y)$ via $q(d)$.
Fix $p\in\mathcal{Q}$ and, for each $d\in\om$, consider the following function $\B_d\colon \om^\om\to \Q$:
\[\mathcal{B}_d(e,C,Y)=
\begin{cases}
\A(\prim_e(C\oplus\Phi_d(Y)))&\mbox{ if $\Phi_d(Y)$ is total,}\\
p&\mbox{ otherwise}.
\end{cases}
\]
Let $\D_d$ be the set of all $(e,C,Y)$ such that $\Phi_d(Y)$ is total.
The continuous function $(e,C,Y)\mapsto \prim_e(C\oplus\Phi_d(Y))$ reduces $\mathcal{B}_d$ to $\mathcal{A}$ on the domain $\mathcal{D}$.
Therefore, by the previous lemma, there is a total Lipschitz function $\hat{\theta}_d$ such that, for all $(e,C,Y)\in \D_d$, $\B_d(e,C,Y)\leq_\Q \A(\hat{\theta}(e,C,Y))$.
Let 
\[
C=\bigoplus_{d\in\om}\hat{\theta}_d.
\]
We claim that $\Bfr^C$ is $(\leq_T,\leq_m)$-UOP.
Given $d$ and $e$, one can effectively find $q(d,e)$ such that 
\[
\hat{\theta}_d(e,C,Y)=\prim_{q(d,e)}(C\oplus Y) \quad (\forall C,Y\in\om^\om).
\]

Let $X\leqt Y$ and suppose $X=\Phi_d(Y)$ for some Turing functional $\Phi_d$.
Since $\Phi_d(Y)$ is total, we then have that
\begin{multline*}
\A(\prim_e(C\oplus X))= \A(\prim_e(C\oplus\Phi_d(Y))) \\
= \mathcal{B}_d(e,C,Y) \leq_\mathcal{Q} \A(\hat{\theta}_d(e,C,Y)) = \A(\prim_{q(d,e)}(C\oplus Y)).
\end{multline*}

Consequently, whenever $X\leq_TY$ via $\Phi_d$, we have $\Bfr^C(\mathcal{A})(X)(e)\leq_\mathcal{Q}\Bfr^C(\mathcal{A})(Y)(q(d,e))$.
In other words, $\Bfr^C(\A)$ is  $(\leq_T,\leq_m)$-UOP, as desired.
\end{proof}

We are now ready to show that $\Afr$ is onto.

\begin{proof}[Proof of Proposition \ref{prop: onto}]
If $\A$ is non-self-dual, let $C$ be as in Lemma \ref{lem:UOP-non-self-dual}, and then we have that $\Bfr^C(\A)$ is $(\leq_T,\leq_m)$-UOP and, by Lemma \ref{lem:onto1}, that $\Afr(\Bfr^C(\mathcal{A}))\equiv_w \A$.

Suppose now that $\A$ is self-dual.
 By Fact \ref{fact:sjr-equal-sd}, $\A$ is $\sigma$-join-reducible, that is, there exists a sequence $\A_0,\ A_1,....$ of non-self-dual functions from $\om^\om$ to $\mathcal{Q}$ such that $\A\equiv_w\bigoplus_n\A_n$.
By Lemma \ref{lem:UOP-non-self-dual}, for each $n$, there is a $C_n\in\om^\om$ such that $\Bfr^{C_n}(\A_n)$ is $(\leq_T,\leq_m)$-UOP, and moreover, the proof of Lemma \ref{lem:UOP-non-self-dual} provides an effective way of computing the witness of the fact that $\Bfr^{C_n}(\A_n)$ is $(\leq_T,\leq_m)$-UOP from given a $n$.
Put $C=\bigoplus C_n$, and then $\Bfr^{C}(\A_n)$ is also $(\leq_T,\leq_m)$-UOP.
We claim that 
\[
\Afr(\bigoplus_n\Bfr^{C}(\A_n))  \equiv_w \A.
\]
On the one hand, we have that 
\[
\Afr(\bigoplus_n\Bfr^{C}(\A_n))(\langle m,e\rangle\fr X)=\A(m\fr \prim_e(C\oplus X)),
\]
 and on the other  that $\A(m\fr X)=\Afr(\bigoplus_n\Bfr^{C}(\A_n))(\langle m,e\rangle\fr X)$, where $e$ is such that $\prim_e(C\oplus X)=X$.

Notice that $\bigoplus_n\Bfr^{C}(\A_n)$ not only is $(\leq_T,\leq_m)$-UOP, but it is also in standard form as needed for Remark \ref{rmk: standard form}.
\end{proof}

\section{The Games and the Embedding Lemma}  \label{sec: games}

\subsection{The Wadge Game $G_w$}

Wadge \cite[Theorem B8]{Wadge83} introduced a perfect-information, infinite, two-player game, known as the {\em Wadge game}, which can be used to define  Wadge reducibility.
For $\mathcal{Q}$-valued functions $\A,\B \colon \om^\om\to\mathcal{Q}$, here is the $\mathcal{Q}$-valued version $G_w(\mathcal{A},\mathcal{B})$ of the Wadge game:
At $n$-th round of the game, Player I chooses $x_n\in\om$ and II chooses $y_n\in\om\cup\{{\sf pass}\}$ alternately (where ${\sf pass}\not\in\om$), and eventually Players I and II produce infinite sequences $X=(x_n)_{n\in\omega}$ and $Y=(y_n)_{n\in\omega}$, respectively.
We write $Y^{\sf p}$ for the result dropping all {\sf pass}es from $Y$.
We say that {\em Player II wins the game $G_w(\mathcal{A},\mathcal{B})$} if
\[\mbox{$Y^{\sf p}$ is an infinite sequence, and }\mathcal{A}(X)\leq_\mathcal{Q}\mathcal{B}(Y^{\sf p}).\]
As in Wadge \cite[Theorem B8]{Wadge83}, one can easily check that $\mathcal{A}\leq_{w}\mathcal{B}$ holds if and only if Player II wins the game $G_w(\mathcal{A},\mathcal{B})$.
Given $\mathcal{Q}^\om$-valued functions $f,g$, we use the abbreviation $G_w(f,g)$ to denote $G_w(\Afr(f),\Afr(g))$, and the same for the rest of the games we define below.

\subsection{The m-Game $\Gm$}

A second version of the Wadge game that  will be useful to us is the game we call $\Gm(\mathcal{A},\mathcal{B})$, where Player II is not allowed to ${\sf pass}$ in his first move, but he can ${\sf pass}$ in  subsequent moves.
In other words, in the game $\Gm(f,g)$, 
Player I plays natural numbers $m,x_0,x_1,\dots$, and Player II plays $n,y_0,y_1,\dots$ alternately, where $n,m, x_0,x_1,...\in\om$ and $y_0,y_1,\dots\in\om\cup\{{\sf pass}\}$.
Player II {\em wins} the game $\Gm(f,g)$ if $Y^{\sf p}$ is infinite and $f(X)(m)\leq_\mathcal{Q}g(Y^{\sf p})(n)$.

\subsection{The Lipchitz-Game $\Gm$}

A third version of the Wadge game that will also be useful to us is the game we call $\Glip(\mathcal{A},\mathcal{B})$, where Player II is not allowed to ${\sf pass}$ at any time.
The rest is all the same.

\subsection{The modified m-Game $\Gtm$}

Steel \cite[Lemma 1]{Steel82} introduced a perfect-information, infinite, two-player game $\Gtm(f,g)$ to study uniformly Turing degree-invariant functions.
Here is a small variation of its $\mathcal{Q}$-valued version:
Alternately, Player I plays natural numbers $m,x_0,x_1,\dots$, and Player II plays $\la n,j\ra ,y_0,y_1,\dots$  with $\la n,j\ra \in\om^2$ and $y_0,y_1,\dots\in\om\cup\{{\sf pass}\}$.
Player II {\em wins} the game $\Gtm(f,g)$ if $Y^{\sf p}$ is infinite and
\[
\Phi^{Y^{\sf p}}_j=X\mbox{ and }f(X)(m)\leq_\mathcal{Q}g(Y^{\sf p})(n),
\]
where $X=(x_n)_{n\in\om}$ and $Y=(y_n)_{n\in\om}$.

\subsection{The plan for embeddability}

The following lemmas lay out the plan to prove the right-to-left direction of Proposition \ref{prop: embeddability}, which states that $\Afr$ is an order-preserving embedding when restricted to $\Afr$-minimal functions.
Recall that the left-to-right direction of Proposition \ref{prop: embeddability} was already proved in Lemma \ref{lem: Afr well defined}.
The lemmas are quite similar in form, except that one assumes that $f$ is $(\leq_T,\leq_m)$-UOP, and the other that $g$ is $(\leq_T,\leq_m)$-UOP.

\begin{lemma}\label{lem: embed f UOP}
Let $f,g\colon\om^\om\to\mathcal{Q}^\om$ be $(\equiv_T,\equiv_m)$-UI,  $\Afr$-minimal functions.
Suppose also that $f$ is $(\leq_T,\leq_m)$-UOP.
Each of the following statements implies the next one:
\begin{enumerate}
\item $\Afr(f)\leq_w \Afr(g)$.    \label{part: embed 1}
\item For every u.p.p.\ tree $S$,  II wins $G_w(f, g \circ S)$.          \label{part: embed 2}
\item For every u.p.p.\ tree $S$,  II wins $\Glip(f, g\circ S)$.\label{part: embed 3}
\item II wins $\Gtm(f, g)$.       \label{part: embed 4}
\item $f\leq\mcone g$.                         \label{part: embed 5}
\end{enumerate}
\end{lemma}

\begin{lemma}\label{lem: embed g UOP}
Let $f,g\colon\om^\om\to\mathcal{Q}^\om$ be $(\equiv_T,\equiv_m)$-UI,  $\Afr$-minimal functions.
Suppose also that $g$ is $(\leq_T,\leq_m)$-UOP and in standard form (as in Remark \ref{rmk: standard form}).
Each of the following statements implies the next one:
\begin{enumerate}
\item $\Afr(f)\leq_w \Afr(g)$.    \label{part: embed 1 g}
\item II wins $G_w(f, g )$.          \label{part: embed 2 g}
\item There is a u.p.p.\ tree $T$ such that  II wins $\Gm(f\circ  T, g )$.\label{part: embed 3 g}
\item $f\leq\mcone g$.                         \label{part: embed 5 g}
\end{enumerate}
\end{lemma}

First, let us see how the lemmas imply the right-to-left direction of Proposition \ref{prop: embeddability}.

\begin{proof}[Proof of Proposition \ref{prop: embeddability}]
Consider $(\equiv_T,\equiv_m)$-UI, {\em $\Afr$-minimal} functions $f,g\colon\om^\om\to\mathcal{Q}^\om$.
The problem is that maybe neither of them is $(\leq_T,\leq_m)$-UOP.
By Proposition \ref{prop: onto}, there is a $(\leq_T,\leq_m)$-UOP function $h$ such that $\Afr(g)\equiv_w \Afr(h)$.
Furthermore, as noted in Remark \ref{rmk: standard form}, we can assume $h$ is in standard form.
We then apply Lemma \ref{lem: embed g UOP} to $f$ and $h$, Lemma \ref{lem: embed f UOP} to $h$ and $g$, and then apply the transitivity of $\leq\mcone$. 
\end{proof}

Let us start by proving the easiest implication in Lemmas \ref{lem: embed f UOP} and \ref{lem: embed g UOP}.
Since $f$ and $g$ are $\Afr$-minimal, we have that $\Afr(f)\leq_w\Afr(g)$ if and only if, for every u.p.p.\ tree $S$, $\Afr(f)\leq_w\Afr(g\circ S)$.
The equivalences between (\ref{part: embed 1}) and (\ref{part: embed 2}) in both lemmas then follow from the equivalence between  $\Q$-Wadge reducibility and the Wadge game.

The implication from (\ref{part: embed 4}) to (\ref{part: embed 5})  follows from the equivalence between  $\leq\mcone$ reducibility and the modified m-game $\Gtm$ (Lemma \ref{lem:embedding2}).

\section{The proof of the embeddability lemmas}

This section is dedicated to proving the rest of Lemmas \ref{lem: embed f UOP} and \ref{lem: embed g UOP}.

\subsection{The case when $f$ is $(\leq_T,\leq_m)$-UOP}

We start with the proof of Lemma  \ref{lem: embed f UOP}.
The implication from (\ref{part: embed 2}) to (\ref{part: embed 3}) in Lemma \ref{lem: embed f UOP} follows from the next lemma and an application of determinacy.

\begin{lemma}\label{lem: the m and lip games}
Let $f \colon \om^\om\to\mathcal{Q}^\om$ be $(\leq_T,\leq_m)$-UOP and $g\colon \om^\om\to\mathcal{Q}^\om$ be $(\equiv_T,\equiv_m)$-UI.
If Player I has a winning strategy for $\Glip(f,g)$, then Player I has a winning strategy for $G_w(f,g)$.
\end{lemma}
\begin{proof}
Let $\tau$ be Player I's strategy in $\Glip(f,g)$.
The difficulty in defining a strategy in $G_w(f,g)$ is that now Player II is allowed to ${\sf pass}$.

Let $\Phi_i$ be a computable operator that removes the 0's from the input, and reduces the rest of the entries by 1.
That is, $\Phi_i(\sigma\fr 0)=\Phi_i(\sigma)$ and $\Phi_i(\sigma\fr (n+1))=\Phi_i(\sigma)\fr n$.
Since $f$ is $(\leq_T,\leq_m)$-UOP, there is a computable function $p$ such that $f(\Phi_i(X))(n)\leq_\mathcal{Q}f(X)(p(n))$ for all $X\in\om^\om$.

We are now ready to describe a winning strategy for Player I in the Wadge game $G_w(f,g)$.
Let $Y=(y_s)_{s\in\om}$ be a sequence produced by Player II in the Wadge game $G_w(f,g)$.
We will play a run of $\Glip(f,g)$ at the same time, where Player II plays $\Yp$.
Let Player I's first move in $G_w(f,g)$ be $x_0=p(n)$, where $n$ is Player I's move in $\Glip(f,g)$.
At any round $s$, if Player II's move $y_s$ is ${\sf pass}$, then let Player I's next move be $x_{s+1}=0$.
If Player II's move is $y_s\not={\sf pass}$, then let Player I follow the winning strategy $\tau$ in the  game $\Glip(f,g)$ and then add $1$, that is, let Player I's next move be $x_{s+1}=\tau(\langle y_0,\dots,y_s\rangle^{\sf p})+1$.

Assume that $(y_s)_{s\in\om}$ contains infinitely many natural numbers; otherwise Player I wins.
If Player I follows the above strategy as we described and plays a sequence $p(n)\fr X$, where $X=\la x_1,x_2,...\ra$, we have $\Phi_i(X)=\tau(Y^{\sf p})^-$ and then we get
\[
\Afr(f)(p(n)\fr X) =f(X)(p(n))\geq_\mathcal{Q}f(\Phi_i(X))(n)=f(\tau(Y^{\sf p})^-)(n)\not\leq_\mathcal{Q}\Afr(g)(Y^{\sf p}).
\]

Consequently, Player I wins the Wadge game $G_w(f,g)$.
\end{proof}

The implication from (\ref{part: embed 3}) to (\ref{part: embed 4}) in Lemma \ref{lem: embed f UOP} follows from the next lemma and an application of determinacy.

\begin{lemma}\label{lem: the lip and mod m games}
Let $f,g \colon \om^\om\to\mathcal{Q}^\om$ be $(\equiv_T,\equiv_m)$-UI functions.
If Player I has a winning strategy for $\Gtm(f,g)$, then Player I has a winning strategy for $\Glip(f,g\circ S)$ for some u.p.p. tree $S$.
\end{lemma}
\begin{proof}
Let $\tau$ be Player I's strategy in $\Gtm(f,g)$.
The difficulty in defining  a strategy in $\Glip(f,g\circ S)$ is that now Player II does not need to play a correct index $e$ to compute Player I's moves.

For each $m, e, Z$, let $n$ and $\theta(m,e, Z)$ be such that $(n, \theta(m,e, Z))$ is Player I's answer to II playing $(\la m, e\ra, Z)$ in $\Gtm(f,g)$.
Let $S\geqt \tau$ be a u.p.p.\ tree.
Then there is a computable operator $\Psi$ such that, for every $Z\in\om^\om$ with $Z\in [S]$, we have $\Psi^Z(m,e)= \theta(m,e, Z)$.
By the Recursion Theorem, there is a computable function $e(m)$ such that $\Phi_{e(m)}^Z=\Psi^Z(m,e(m))$.

To define Player I's strategy in $\Glip(f, g\circ S)$ answering to Player II moving $(m,Y)$, all we have to do is imitate  Player I's strategy in $\Gtm(f, g)$ answering to Player II moving $(\la m,e(m)\ra, S[Y])$.
Notice that since Player II is not allowed to ${\sf pass}$, $Y=\Yp$, and hence $S[Y]$ computes Player I's moves using $\Phi_{e(m)}$.
\end{proof}

The implication from (\ref{part: embed 4}) to (\ref{part: embed 5}) follows from the next lemma.

\begin{lemma}\label{lem:embedding2}
Let $f,g \colon \om^\om\to\mathcal{Q}^\om$ be $(\equiv_T,\equiv_m)$-UI functions.
If Player II has a winning strategy for $\Gtm(f,g)$, then $f\leq\mcone g$.
\end{lemma}
\begin{proof}
Consider a winning strategy for Player II in  $\Gtm(f,g)$.
Suppose the answer to Player I playing $n\fr X$ is Player II playing $\la m,j\ra\fr Y$.
From the strategy, we get a function $\psi$ that outputs $m$ given $n$ and satisfies $f(X)(n)\leq_\Q g(Y)(\psi(n))$ for all $n\in\om$ and $X\in \om^\om$.
Also, if we take $X$ that can compute the strategy, we get, for each $n$, an index $i(n)$ for the Turing equivalence between $X$ and $\Yp$: $X$ computes $\Yp$ using $n$ and the strategy, and $\Yp$ computes $X$ using $\Phi_j$.
Thus, 
\[
f(X)(n)\leq_\mathcal{Q} g(\Yp)(\psi(n))
\leq_\mathcal{Q} g(X)(\Phi_{u(i(n))}\circ\psi(n)),
\]
where $u$ witnesses that $g$ is $(\equiv_T,\equiv_m)$-UI.
This implies that  $f(X)\leq_m g(X)$ whenever $X\in \om^\om$ computes Player II's strategy.
\end{proof}

This finishes the proof of Lemma \ref{lem: embed f UOP}.

\subsection{The case when $g$ is $(\leq_T,\leq_m)$-UOP}   \label{sec: wadge to m games}

We now concentrate on the proof of Lemma  \ref{lem: embed g UOP}.
The implication from (\ref{part: embed 3 g}) to (\ref{part: embed 5 g})  follows from the next lemma.

\begin{lemma}\label{lem:embedding3}
Let $f \colon \om^\om\to\mathcal{Q}^\om$ be $(\equiv_T,\equiv_m)$-UI and $g\colon \om^\om\to\mathcal{Q}^\om$ be $(\leq_T,\leq_m)$-UOP.
If there is a u.p.p\ tree $T$ such that Player II has a winning strategy for $\Gm(f\circ T,g)$, then $f\leq\mcone g$.
\end{lemma}
\begin{proof}
The proof is very similar to that of Lemma \ref{lem:embedding2}, with the exceptions that now we do not need to use that $\Yp$ computes $X$, and that we need to consider the tree $T$.
 
Consider a winning strategy for Player II in  $\Gm(f\circ T,g)$.
Suppose the answer to $(n,X)$ is $(m, Y)$.
From the strategy, we get a function $\psi$ that outputs $m$ given $n$ and satisfies $f(X)(n)\leq_\Q g(Y)(\psi(n))$ for $n\in\om$ and $X\in [T]$.
If we take $X\in [T]$ that can compute the strategy, then $X$ can compute $\Yp$ uniformly using $n$.
Let  $i(n)$ be an index for the Turing reduction from $\Yp$ to $X$.
Thus, 
\[
f(X)(n)\leq_\mathcal{Q} g(\Yp)(\Psi(n))
\leq_\mathcal{Q} g(X)(\Phi_{u(i(n))}\circ\Psi(n)),
\]
where $u$ witnesses that $g$ is $(\leq_T,\leq_m)$-UOP, and hence $f(X)\leq_m g(X)$ for all $X\in [T]$ that compute the strategy.
Now, if we take any $X\geqt T$, we have that $X\teq T[X]$, and hence that $f(X)\leq_m f(T[X])$ and $g(T[X])\leq_m g(X)$, since $f$ and $g$ are $(\equiv_T,\equiv_m)$-UI.
Putting all this together, we get $f(X)\leq_m g(X)$, for all $X$ that compute $T$ and the strategy.
This shows that $f\leq\mcone g$.
\end{proof}

All that is left to finish the proof of Lemma \ref{lem: embed g UOP} is to prove  (\ref{part: embed 2}) implies (\ref{part: embed 3}), connecting the Wadge game and the $\mathbf{m}$-game.
This will then finish the proofs of Proposition \ref{prop: embeddability} and our main theorems.
The proof is divided in two cases:
the case when $\Afr(g)$ is $\sigma$-join-irreducible, and the case when $\Afr(g)$ is  $\sigma$-join reducible and $g$ is in standard form (by Fact \ref{fact:sjr-equal-sd} and Remark \ref{rmk: standard form}). 
The existence of the u.p.p.\ tree $T$ mentioned in (\ref{part: embed 3}) is only needed in the latter case.

\begin{lemma}\label{lem:UI-NSD-mc}
Let $\A$ and $\B$ be $\mathcal{Q}$-valued functions on $\om^\om$ such that $\B$ is $\sigma$-join-irreducible.
If Player II has a winning strategy for $G_w(\A,\B)$, then Player II has a winning strategy for $G_\mathbf{m}(\A,\B)$.
\end{lemma}

\begin{proof}
By Fact \ref{fact:sji-deg-pres}, if $\B$ is $\sigma$-join-irreducible, there is $Z\in\om^\om$ such that $\B\leq_w\B\upto[Z\upto n]$ for any $n\in\om$.
In particular, Player II has a winning strategy $\tau$ for $G_w(\A,\B\upto[Z(0)])$.
In the game $G(\A,\B)$, Player II plays $Z(0)$, and then follows $\tau$.
This clearly gives II's winning strategy for $G_\mathbf{m}(\A,\B)$.
\end{proof}

We now move to the last case of $\Afr(g)$ being  $\sigma$-join reducible.
We say that a closed set $P\subseteq 2^\om$ is {\em thin} if, for every $\Pi^0_1$ set $Q\subseteq 2^\om$, the intersection $P\cap Q$ is clopen in $P$.
We also say that a closed set $P\subseteq 2^\om$ is {\em almost thin} if there are at most finitely many $X\in 2^\om$ such that $P\cap[X\upto n]$ is not thin for any $n\in\om$.
Here, $X\upto n$ is the unique initial segment of $X$ of length $n$, and for a finite string $\sigma$, $[\sigma]$ is the set of all reals extending $\sigma$.
For a number $k\in\om$, we also use $[k]$ to denote $[\la k\ra]$.

Cenzer, Downey, Jockusch, and Shore \cite[Theorem 2.10]{CDJS} showed that an element $X$ of a thin $\Pi^0_1$ class satisfies that $X'\leq_TX\oplus\emptyset''$.
We extend their result as follows:

\begin{lemma}\label{lem:CDJS}
Let $T\subseteq 2^{<\om}$ be a tree such that $[T]$ is almost thin.
Then, for every $X\in[T]$, either $X'\leq_TX\oplus T''$ or $X\leq_TT''$ holds.
\end{lemma}

\begin{proof}
We first claim that if $[T]$ is thin, then $X'\leq_TX\oplus T''$ for any $X\in[T]$.
Given $e$, let $Q_{e}$ be the $\Pi^0_1$ set consisting of oracles $X\in 2^\om$ such that $\Phi^X_e(e)$ diverges.
Since $[T]$ is thin, $Q_{e}\cap[T]$ is clopen in $[T]$.
Therefore, there is a height $h(e)$ such that $\Phi_e^X(e)$ converges if and only if $\Phi_e^{X\upto h(e)}(e)$ converges for every $X\in[T]$.
Note that such $h$ can be computed from $T''$ by searching for the smallest $h(e)$ such that if $\sigma$ is an extendible node of $T$ of length $h(e)$, and $\Phi_e^{\tau}(e)$ converges for some node $\tau\succeq\sigma$ in $T$, then $\Phi_e^\sigma(e)$ already converges.
This shows that $X'\leq_TX\oplus T''$ for every $X\in[T]$.

Now, let us assume $T$ is almost thin.
If $X\in [T]$ satisfies that $[T]\cap[X\upto n]$ is thin for some $n$, then we can apply the previous argument to the closed set $[T]\cap[X\upto n]$ and obtain that $X'\leq_TX\oplus T''$.
There are finitely many $X$'s for which $[T]\cap[X\upto n]$ is not thin for any $n$.
Again, by restricting ourselves to a tree of the form $[T]\cap[X\upto n]$, let us assume $X$ is the only path in $[T]$ for which $[T]\cap[X\upto n]$ is not thin for any $n$.
We will show that $X\leqt T''$.

Let $Q\subseteq 2^{<\om}$ be a computable tree witnessing that $T$ is not thin; i.e., such that $[Q]\cap[T]$ is not clopen in $[T]$.
Let $S\subseteq 2^{<\om}$ be the set of strings $\si\in T$ such that $[Q]\cap[T]\cap [\si]$ is not clopen in $[T]\cap[\si]$.
First, let us observe that $X$ is the only path through $S$:
$S$ must have some path, as otherwise there is some $\ell$ such that, for all $\si\in 2^\ell$, $[Q]\cap[T]\cap [\si]$ is clopen in $[T]\cap[\si]$, and hence $[Q]\cap[T]$ would be clopen in $[T]$.
Suppose $Y\in [T]$, but $Y\neq X$.
Then there is some $n$ such that $[T]\cap[Y\upto n]$ is thin, and hence $[Q]\cap[T]\cap [Y\upto n]$ is clopen in $[T]\cap[Y\upto n]$.
Thus, $Y\upto n\not\in S$.
It follows that $X$ is the only path through $S$.

Second, let us observe that $S$ is $\Pi^0_1$ relative to $T''$:
A string $\si$ is not in $S$ if and only if there exist $\ell\geq |\si|$ such that, for every $\tau\in 2^\ell$ extending $\si$, either $\tau\not\in Q$ (and hence $[Q]\cap[T]\cap[\tau]=\emptyset$), or every $\gamma\in T$ which extends $\tau$ and extendible in $T$ belongs to $Q$ too (and hence $[Q]\cap[T]\cap[\tau]=[T]\cap[\tau]$).

Since $X$ is the only path on a $\Pi^0_1$ class relative to $T''$, we get that $X\leqt T''$.
\end{proof}

\begin{lemma}\label{lem:UI-sji-win}
Let $f \colon \om^\om\to\mathcal{Q}^\om$ be a $(\equiv_T,\equiv_m)$-UI function, and $g \colon \om^\om\to\mathcal{Q}^\om$ be a $(\leq_T,\leq_m)$-UOP function such that $\Afr(g)$ is $\sigma$-join-reducible and $g$ is in standard form.
If Player II wins $G_w(f,g)$, then, for some u.p.p.\ tree $T$, Player II has a winning strategy for $G_\mathbf{m}(f\circ T,g)$.
\end{lemma}

\begin{proof}
Since $g$ is in standard form, we have that $g$ is of the form $\bigoplus_{n\in\om}g_n$, where $\Afr(g_n)$ is $\sigma$-join-irreducible.
By Fact \ref{fact:sji-deg-pres}, there are $z_n\in\om$ such that $\Afr(g_n)\leq_w\Afr(g_n)\upto[z_n]$ since $\Afr(g_n)$ is $\sigma$-join-irreducible.

We say that a subset $D$ of a quasi-order $\mathcal{P}$ is {\em directed} if for any $p,q\in D$, there is $r\in D$ such that $p,q\leq_\mathcal{P} r$.
By the Erd\"os-Tarski theorem \cite{ET43}, if $\mathcal{P}$ has no infinite antichains, then $\mathcal{P}$ is covered by a finite collection $(D_m)_{m<l}$ of directed sets.
We now consider the quasi-order $\leq_{\om}$ on $\om$ defined by $m\leq_{\om} n$ if and only if $\Afr(g_m)\leq_w\Afr(g_n)$.
Since $(\om;\leq_{\om})$ is bqo, it is covered by finitely many directed sets $(D_m)_{m<l}$.

Given numbers $m$ and $n$, consider the following closed set:
\[\F_{m,n}=\{X\in 2^\om:(\forall i\in D_m)(\forall k\in\om)[\Afr(f)\upto[n\fr X\upto k]\not\leq_w\Afr(g_i)]\}.\]

Let $C\geq_T\bigoplus_mD_m$ be a sufficiently powerful oracle deciding whether $\Afr(f)\upto[n\fr\tau]\not\leq_w\Afr(g_i)$ given $n,i\in\om$ and $\tau\in 2^{<\om}$.
In particular, we have that 
\[\F_{m,n}\mbox{ is }\Pi^0_1(C).\]

\noindent
{\bf Case 1.}
For all $n\in\om$, there is $m<l$ such that $\F_{m,n}$ is almost thin.

\medskip

In this case, by Lemma \ref{lem:CDJS}, every element $X\in\F_{m,n}$ satisfies $X'\leq_TX\oplus C''$ or $X\leq_TC''$.
Thus, no $X$ with $X>_T C''$ belongs to $\F_{m,n}$.
Let $K$ be the compact set $\{X\oplus C''': X\in 2^\om\}$.
Since $K$ is disjoint from $\F_{m,n}$, for every $X\in K$, there are $i\in D_m$ and $k\in\om$ such that $\Afr(f)\upto[n\fr X\upto k]\leq_w\Afr(g_i)$.
By compactness of $K$, such an $i$ can be chosen from a finite set $E\subseteq D_m$.
Since $D_m$ is directed, there is $i(n)\in D_m$ such that $e\leq_{\om}i(n)$ for any $e\in E$.
Let $T$ be a u.p.p.\ tree such that the image of $2^\om$ is inside $K$.

We now claim that Player II has a winning strategy for the game $G_\mathbf{m}(f\circ T^\ast,g)$.
If Player I's first move is $n$, Player II chooses a pair $\la i(n),z_{i(n)}\ra$.
Given Player I's move $X$, Player II waits for a round $s$ such that $\Afr(f)\upto[n\fr T^\ast[X]\upto s]\leq_w\Afr(g_{i(n)})$.
Such $s$ exists by our choice of $i(n)$.
By the definition of $z_{i(n)}$, we have $\Afr(f)\upto[n\fr T^\ast[X]\upto s]\leq_w\Afr(g_{i(n)})\upto [z_{i(n)}]$, and then Player II follows a winning strategy witnessing this.
This procedure gives a desired winning strategy for Player II.

\medskip

\noindent
{\bf Case 2.}
Otherwise, there is $n\in\om$ such that $\F_{m,n}$ is not almost thin for any $m<l$.

\medskip

In this case,  there is a sequence of different reals $(X_m)_{m<l}$ such that $\F_{m,n}\cap [X_m\upto k]$ is not thin for any $k$.
Therefore,  there is a sequence $(\sigma_m)_{m<l}$ of pairwise incomparable strings such that $\F_{m,n}\cap[\sigma_m]$ is not thin for any $m<l$:

For each $m<l$, let $Q_m$ be a computable tree witnessing that $\F_{m,n}\cap[\sigma_m]$ is not thin.
Let $(\tau_k^m)_{k\in\om}$ be the set of minimal strings extending $\si_m$, not in $Q_m$.
Thus, for each $m<l$, $(\tau_k^m)_{k\in\om}$ is a computable sequence  of pairwise incomparable strings extending $\sigma_m$ such that $\tau_k^m$ is extendible in $\F_{m,n}$ for infinitely many $k\in\om$.
Since $\tau^m_k$ is incomparable with $\tau^i_j$ whenever $(i,j)\not=(m,k)$, there is a fixed pair $(d,e)$ of indices of computable functions witnessing $0^{lk+m}1\fr X\equiv_T\tau^m_k\fr X$.
Let $u$ witness that $f$ is $(\equiv_T,\equiv_m)$-UI, and then we have
\[f(\tau^m_k\fr X)(n)\leq_\mathcal{Q}f(0^{lk+m}1\fr X)(\Phi_{u(d,e)}(n)).\]

We claim that $\Afr(f)\not\leq_w\Afr(g)$ (i.e., that I wins $G_w(f,g)$), showing that  case 2 was not possible to begin with.
Player I first chooses $\Phi_{u(d,e)}(n)$.
Then Player I plays along $0^\om$ until Player II moves to some $\la i,y_0\ra\not={\tt pass}$ at some round $s$.
Let $m$ be such that $i\in D_m$.
Player I searches for a large $k$ so that $s\leq lk+m$ and that $\tau^m_k$ is extendible in $\F_{m,n}$.
Then, $\Afr(f)\upto[n\fr\tau^m_k]\not\leq_w\Afr(g_i)$, since $i\in D_m$ ,and therefore, Player I has a winning strategy for the game $G_w(\Afr(f)\upto[n\fr\tau^m_k],\Afr(g_i))$.
In this game, given Player II's play $Y=(y_n)_{n\in\om}$, Player I's winning strategy yields a play of the form $(n,\tau^m_k\fr \theta(Y))$.
Then I's play $\Phi_{u(d,e)}(n)\fr 0^{lk+m}1\fr\theta(Y)$  in the original game clearly gives a winning strategy.
\end{proof}

\bibliography{manyone}
\bibliographystyle{alpha}

\end{document}